\documentclass[11pt]{amsart} 
\usepackage[a4paper,margin=2.5cm]{geometry}
\usepackage{amsmath,amssymb,amsthm,enumitem}
\usepackage{xcolor}
\usepackage[
    colorlinks=true,
    linkcolor=blue,
    citecolor=blue,
    urlcolor=blue!60
]{hyperref}

\IfFileExists{newtxtext.sty}{\usepackage{newtxtext}}{}
\IfFileExists{newtxmath.sty}{\usepackage{newtxmath}}{}

\newtheorem*{theorem*}{Theorem}
\newtheorem{theorem}{Theorem}
\newtheorem{proposition}[theorem]{Proposition}
\newtheorem{lemma}[theorem]{Lemma}
\newtheorem{corollary}[theorem]{Corollary}

\theoremstyle{remark}
\newtheorem{remark}[theorem]{Remark}

\setlist[enumerate,1]{
  label=\textup{(\roman*)},
  align=left,
  leftmargin=0.5em,
  labelsep=0.5em,
  itemsep=2pt,
  topsep=2pt,
  partopsep=0pt,
  parsep=0pt
}

\renewcommand{\pmod}[1]{\allowbreak\mkern4mu(\mathrm{mod}\,#1)}
\DeclareMathOperator{\SubF}{SubF}

\title{Subindices of two-element subsets of finite groups}

\author{M.~H.~Hooshmand}
\address{Department of Mathematics, Shi.C.,
Islamic Azad University, Shiraz, Iran}
\email{MH.Hooshmand@iau.ac.ir, hadi.hooshmand@gmail.com}

\author{Mikhail Kabenyuk}
\address{Kemerovo State University, Russia}
\email{kabenyuk@gmail.com, kabenyuk@kemsu.ru}

\subjclass[2020]{Primary 20D60; Secondary 20E34}
\keywords{finite groups, subfactors of groups, subindices of subsets, index stability, $2$-index stability}

\date{}

\begin{document}

\begin{abstract}
Let $A$ be a nonempty subset of a group $G$.
A \emph{right subfactor} relative to $A$ is a subset $B$
maximal under inclusion among those for which every element
of $AB$ has a unique representation as $ab$, with $a\in A$
and $b\in B$.
Left subfactors are defined analogously using $BA$.
The infima and suprema of the cardinalities of right and left
subfactors define four \emph{subindices}.
A subset is called \emph{index-stable} if its four subindices coincide.

For any infinite group $G$, every nonempty finite subset of $G$ is index-stable.

For finite groups, the situation is more complicated.
We prove that for a two-element subset of a finite group,
the left and right subindices coincide pairwise.
Let $A=\{x,y\}$ with $x\ne y$, and let $X$ be the cyclic subgroup
generated by $x^{-1}y$.
Then the lower subindex of $A$ is $|G:X|\cdot\lceil |X|/3\rceil$,
and the upper subindex is $|G:X|\cdot\lfloor |X|/2\rfloor$.
These formulas imply that $A$ is index-stable if and only if
$|X|$ is one of $2,3,4,5,7$.

For any two-element subset of a finite group $G$, the lower
subindex is at least $\lceil |G|/3\rceil$, and the upper subindex
is at most $\lfloor |G|/2\rfloor$.
We establish necessary and sufficient conditions for equality
in each of these bounds.

All two-element subsets of a finite group are index-stable
if and only if every element of the group has order
$1$, $2$, $3$, $4$, $5$, or $7$.
We give a complete structural classification of such groups.
\end{abstract}

\maketitle

\section{Introduction}
Factorizations of a group into direct products of two subsets form a classical
topic in group theory, with several related subjects and applications (see,
e.g., \cite{Hoo-subindex, Stein}). A challenging open question concerning the
factorization of finite groups by subsets was posed in 2014 and later recorded
as Problems~19.35 and~20.37 of the Kourovka Notebook \cite{Kouk-Mazuk}. Every
subgroup is a factor in this sense, but there are factors that are not
subgroups. A subsemigroup is a factor if and only if it is a subgroup (see
\cite{Hoo-finite}), and a subset whose size does not divide $|G|$ can never be a
factor of a finite group $G$. It is therefore natural to ask whether the notion
of index can be extended from subgroups to factors, or more generally to
arbitrary subsets of a group.

In \cite{Hoo-subindex}, the first author introduced,
for an arbitrary subset $A$ of a group $G$, the notion of a \emph{right subfactor} relative to $A$:
a subset $B\subseteq G$ maximal under inclusion subject to the product $AB$ being direct.
The infimum and supremum of the cardinalities of right subfactors
define the \emph{lower and upper right subindices} $|G:A|^-$ and $|G:A|^+$, respectively.

The \emph{left subfactors} and the corresponding \emph{lower and upper left subindices} $|G:A|_-$ and $|G:A|_+$
are defined analogously, with $BA$ in place of $AB$.
Since $|G:A|_{\pm}=|G:A^{-1}|^{\pm}$ and $|A^{-1}|=|A|$,
it suffices to study the right subindices.
All four subindices equal the classical index when $A$ is a subgroup.
A subset is called \emph{index-stable} if its four subindices coincide,
and a group is index-stable if all of its subsets are.

Alternative terminology is used in \cite{Kab-finite,Kab-infinite}:
\emph{s-factor} instead of \emph{subfactor}, \emph{index}
instead of \emph{subindex}, and \emph{stable}
instead of \emph{index-stable}.

Besides its group-theoretic content, the
theory turned out to have unexpected connections with additive combinatorics and
number theory; for example, the classical conjecture that every even number is a difference of
two primes is equivalent to the statement that the set of primes is
index-stable in $\mathbb Z$ with index $2$ \cite{Hoo-subindex, Hoo-infinite}.

The case of finite groups was developed in \cite{Hoo-finite}, where sharp
bounds for the subindices were obtained, a first algorithm for finding
subfactors was given, and it was shown that every prime divisor of the order of a finite
index-stable group is one of $2$, $3$, $5$, or $7$.
The central problem, the classification of
finite index-stable groups, was solved in \cite{Hoo-Yousef}:
with the help of a computer-assisted analysis of small groups
combined with theoretical arguments, they proved that, up to isomorphism, there
are exactly $14$ finite index-stable groups.

A different view was introduced by the second author \cite{Kab-finite}.
He observed that the right subfactors relative to $A$ are precisely the maximal independent
sets of the Cayley graph $\operatorname{Cay}(G,A^{-1}A\setminus\{1\})$. Hence
the subindices $|G:A|^-$ and $|G:A|^+$ are the independent domination number and
the independence number of this graph, and the right index stability of $A$
means exactly that the graph is \emph{well-covered}. This translation allowed
him to give a proof of the classification of the $14$ groups that does not use
a computer.

For infinite groups the situation is completely different: no infinite group
is index-stable, as was proved independently in \cite{Kab-infinite} and \cite{Hoo-infinite};
this answers Problem~21.58 of the Kourovka Notebook \cite{Kouk-Mazuk}.
Finally, for cyclic groups, the exact subindices of intervals were determined in \cite{Hoo-Zn}.
The same work shows that if a finite group $G$ is right
$k$-index-stable for $k\ge4$, then every prime divisor
$p$ of $|G|$ satisfies $p\le3k-1$. For $k=2$
(resp. $k=3$), the prime divisors of $|G|$
belong to $\{2,3,5,7\}$ (resp. $\{2,3,5,7,11\}$).

Since index-stable groups are so rare, it is natural to weaken the condition.
A finite group $G$ is called \emph{$k$-index-stable} if every $k$-subset of $G$
is index-stable. Every group is $1$-index-stable, because the subindices of a
singleton all equal $|G|$, and the index-stable groups are those that are
$k$-index-stable for all relevant $k$. The classification of finite
$k$-index-stable groups is therefore a natural intermediate problem, and it is
open in general. The first nontrivial case is $k=2$, and this is the subject of
the present paper. Subsets with two elements are the smallest subsets for which
the subindices are nontrivial, and they turn out to be completely computable.

Indeed, if $A=\{a,b\}$, then $A=a\{1,x\}$ with $x=a^{-1}b$, so, by the reduction
to a subgroup, all four subindices of $A$ are determined by the cyclic subgroup
$\langle x\rangle\cong\mathbb Z_m$, where $m=\operatorname{ord}(x)$. They equal
$\frac{|G|}{m}\lceil m/3\rceil$ and $\frac{|G|}{m}\lfloor m/2\rfloor$ (the
independent domination number and the independence number of the cycle graph
$C_m$ are $\lceil m/3\rceil$ and $\lfloor m/2\rfloor$), and the left and right
subindices coincide. In particular, $\{a,b\}$ is index-stable if and only if
$\operatorname{ord}(a^{-1}b)\in\{2,3,4,5,7\}$, and consequently a finite group
is $2$-index-stable if and only if all of its element orders lie in
$\{1,2,3,4,5,7\}$. Such groups have all elements of prime-power order; the
solvable ones are governed by Higman's theorem \cite{Higman}, and the
nonsolvable ones by the Cameron--Maslova classification of groups with
edgeless Gruenberg--Kegel graph \cite{Cameron-Maslova}. Combining these results
gives a complete structural classification of finite $2$-index-stable groups.

The paper is organized as follows. In Section~2 we recall the basic notions,
the general bounds, and the reduction to a subgroup. In Section~3 we derive the
exact subindices of $2$-subsets and characterize prescribed upper subindices.
Section~4 characterizes index stability of $2$-subsets. In Section~5 we
determine when a $2$-subset attains the universal bounds, and when it is
strongly index-unstable, including the cyclic case. Section~6 contains the
classification of finite $2$-index-stable groups.

\section{Preliminaries}

Let $G$ be a group and let $A,B\subseteq G$ be nonempty subsets. We say that
the product $AB$ is \emph{direct}, and write $AB=A\cdot B$, if every element
of $AB$ has a unique representation in the form $ab$, where $a\in A$ and
$b\in B$. Equivalently,
\[
A^{-1}A\cap BB^{-1}=\{1\}.
\]

A \emph{right subfactor} of $G$ relative to $A$ is a subset $B\subseteq G$
which is maximal with respect to inclusion among the subsets satisfying
$AB=A\cdot B$. The family of all right subfactors of $G$ relative to $A$ is
denoted by $\SubF_r(G:A)$. Left subfactors and the family
$\SubF_\ell(G:A)$ are defined analogously.

For finite $G$, the four subindices of $A$ are
\[
|G:A|^-=\min\{|B|:B\in\SubF_r(G:A)\},
\quad
|G:A|^+=\max\{|B|:B\in\SubF_r(G:A)\},
\]
and
\[
|G:A|_-=\min\{|B|:B\in\SubF_\ell(G:A)\},
\quad
|G:A|_+=\max\{|B|:B\in\SubF_\ell(G:A)\}.
\]

\begin{remark}
For infinite $G$, the four subindices are defined analogously,
using infima and suprema instead of minima and maxima.
Every subset $A\subseteq G$ with $|A|<|G|$ then has all four
subindices equal to $|G|$ (see \cite{Hoo-infinite}).
In particular, this holds for every two-element subset of $G$.

Henceforth, all groups under consideration are assumed to be finite,
and all subsets under consideration are assumed to be nonempty.
\end{remark}

We say that $A$ is \emph{right index-stable} if
\[
|G:A|^-=|G:A|^+,
\]
and \emph{left index-stable} if
\[
|G:A|_-=|G:A|_+.
\]
The subset $A$ is \emph{index-stable} if all four subindices are equal.

\subsection{General bounds}

For every nonempty subset $A$ of a finite group $G$,
\[
\left\lceil\frac{|G|}{|A^{-1}A|}\right\rceil
\le |G:A|^-\le |G:A|^+
\le\left\lfloor\frac{|G|}{|A|}\right\rfloor.
\]
We refer to the first and last terms in these inequalities as the
\emph{lower bound} and the \emph{upper bound}, respectively
\cite{Hoo-finite}.

\subsection{Reduction to a subgroup}

The four subindices are invariant under translations of $A$. In particular,
for $g,h\in G$, the corresponding subindices of $A$ and $gAh$ are equal
\cite{Hoo-subindex, Hoo-infinite}.

Let $H\le G$ and let $A\subseteq H$. Then the subindices of $A$ in $G$ are
obtained from those in $H$ by multiplication by the index $|G:H|$ (see \cite{Hoo-infinite, Kab-finite}):
\[
|G:A|^\pm=|G:H||H:A|^\pm.
\]

This reduction will be used repeatedly for $2$-subsets. Indeed, if
$A=\{a,b\}$ and
\[
x=a^{-1}b,
\]
then
\[
a^{-1}A=\{1,x\}\subseteq\langle x\rangle.
\]
Thus the computation of the four subindices of $A$ reduces to the
corresponding computation for the $2$-subset $\{1,x\}$ of the cyclic
subgroup $\langle x\rangle$.

\section{Exact subindices of 2-subsets}

Throughout this section, let $G$ be a finite group of order $n$, and let
$A=\{a,b\}$, where $a\ne b$. Put
\[
x=a^{-1}b,\quad m=\operatorname{ord}(x),\quad
H=\langle x\rangle,\quad q=|G:H|=\frac{n}{m}.
\]

\subsection{Reduction to a cyclic subgroup}

Since
\[
A=a\{1,x\},
\]
translation invariance and the subgroup reduction give
\[
|G:A|^\pm=q|H:\{1,x\}|^\pm.
\]
The subgroup $H$ is cyclic of order $m$, and under the natural identification
$H\cong\mathbb Z_m$, the subset $\{1,x\}$ corresponds to $\{0,1\}$.
Hence
\[
|H:\{1,x\}|^-=\left\lceil\frac{m}{3}\right\rceil,
\quad
|H:\{1,x\}|^+=\left\lfloor\frac{m}{2}\right\rfloor.
\]

This is the case $A=\{0,1\}$ in $\mathbb Z_m$; see \cite{Hoo-Yousef, Kab-finite,
Hoo-Zn}. (For $m=2$ we have $\{1,x\}=H$, and both values equal $1$ trivially.)

For the left subindices, put
\[
y=ab^{-1}.
\]
Then
\[
Ab^{-1}=\{1,y\},
\]
and $y$ is conjugate to $x^{-1}$. In particular,
\[
\operatorname{ord}(y)=\operatorname{ord}(x)=m.
\]
Applying the same argument to the cyclic subgroup $\langle y\rangle$, we
obtain the same formulas for the left subindices.

\begin{theorem}\label{thm:exact-subindices}
Let $G$ be a finite group of order $n$, and let $A=\{a,b\}$ with $a\ne b$.
Put
\[
m=\operatorname{ord}(a^{-1}b),
\quad
q=\frac{n}{m}.
\]
Then
\[
|G:A|_-=|G:A|^-=q\left\lceil\frac{m}{3}\right\rceil,
\]
and
\[
|G:A|_+=|G:A|^+=q\left\lfloor\frac{m}{2}\right\rfloor.
\]
In particular, the four subindices of a $2$-subset are completely determined
by $n$ and $\operatorname{ord}(a^{-1}b)$.
\end{theorem}

In what follows, unless otherwise stated, we consider only the right
subindices of $2$-subsets, which we simply call the lower and upper
subindices. By Theorem~\ref{thm:exact-subindices}, the corresponding
left subindices have the same values for finite groups.

\subsection{Exact and residue formulas}
Theorem~\ref{thm:exact-subindices} can be rewritten according to the residue
of $m$ modulo $2$ and modulo $3$.
For the upper subindices,
\[
|G:A|^+=
\begin{cases}
\dfrac{n}{2}, & m\equiv0\pmod2,\\[2mm]
\dfrac{n-q}{2}, & m\equiv1\pmod2.
\end{cases}
\]
For the lower subindices,
\[
|G:A|^-=
\begin{cases}
\dfrac{n}{3}, & m\equiv0\pmod3,\\[2mm]
\dfrac{n+2q}{3}, & m\equiv1\pmod3,\\[2mm]
\dfrac{n+q}{3}, & m\equiv2\pmod3.
\end{cases}
\]

We shall use these formulas repeatedly below.

It is also useful to record the size of the difference sets. Since
\[
A^{-1}A=\{1,x,x^{-1}\},
\]
we have
\[
|A^{-1}A|=
\begin{cases}
2, & m=2,\\
3, & m\ge3.
\end{cases}
\]

\subsection{Prescribed upper subindices}

The exact formula also gives a simple criterion for the existence of a
$2$-subset with a prescribed upper subindex.

\begin{proposition}\label{prop:prescribed-upper}
Let $G$ be a finite group of order $n$, and let
$1\le i\le\lfloor n/2\rfloor$.

If $i<n/2$, then $G$ contains a $2$-subset $A$ such that
\[
|G:A|^+=i
\]
if and only if
\[
m=\frac{n}{n-2i}
\]
is an odd integer and $G$ contains an element of order $m$.

If $i=n/2$, then such a $2$-subset exists if and only if $G$ contains an
element of even order.
\end{proposition}

\begin{proof}
Let $A=\{a,b\}$ and put
\[
m=\operatorname{ord}(a^{-1}b),
\quad
q=\frac{n}{m}.
\]
If $m$ is even, then Theorem~\ref{thm:exact-subindices} gives
\[
|G:A|^+=\frac{n}{2}.
\]
Thus the case $i<n/2$ requires $m$ to be odd. In this case,
\[
i=\frac{n-q}{2},
\]
and hence
\[
q=n-2i,
\quad
m=\frac{n}{q}=\frac{n}{n-2i}.
\]
Conversely, if $G$ contains an element $x$ of this odd order, then for
$A=\{1,x\}$ we have $q=n-2i$, and therefore
\[
|G:A|^+=\frac{n-q}{2}=i.
\]
The case $i=n/2$ follows directly from the even case of the residue formula.
\end{proof}

As a first application, we recover the following characterization.

\begin{corollary}\label{cor:upper-index-three}
A finite group $G$ has a $2$-subset with upper subindex $3$ if and only if
\[
|G|\in\{6,7,9\}.
\]
\end{corollary}

\begin{proof}
Let $n=|G|$. Necessarily $n\ge6$. If $3=n/2$, then $n=6$.

Now suppose that $3<n/2$. By Proposition~\ref{prop:prescribed-upper},
\[
m=\frac{n}{n-6}
\]
must be an odd integer. Equivalently,
\[
m=1+\frac{6}{n-6}.
\]
Thus $n-6$ divides $6$. The possibilities are
\[
n\in\{7,8,9,12\},
\]
giving respectively
\[
m\in\{7,4,3,2\}.
\]
The odd possibilities are therefore $n=7$ and $n=9$.

Conversely, a group of order $6$ contains an element of even order, a group
of order $7$ contains an element of order $7$, and a group of order $9$
contains an element of order $3$. Hence in each of these cases a $2$-subset
with upper subindex $3$ exists.
\end{proof}

\section{Index stability of 2-subsets}
In this section we characterize index stability of $2$-subsets.

\begin{theorem}\label{thm:2subset-stability}
Let $G$ be a finite group and let $A=\{a,b\}$ with $a\ne b$. Then the
following conditions are equivalent:
\begin{enumerate}
\item $A$ is right index-stable;
\item $A$ is left index-stable;
\item $A$ is index-stable;
\item
$
\operatorname{ord}(a^{-1}b)\in\{2,3,4,5,7\}.
$
\end{enumerate}
\end{theorem}

\begin{proof}
By Theorem~\ref{thm:exact-subindices}, each of the first three conditions is
equivalent to
\[
\left\lceil\frac{m}{3}\right\rceil
=
\left\lfloor\frac{m}{2}\right\rfloor,
\]
where $m=\operatorname{ord}(a^{-1}b)\ge2$. The solutions of this equation are
precisely $m\in\{2,3,4,5,7\}$.
\end{proof}

Thus index stability of a $2$-subset depends only on the order of the
difference $a^{-1}b$, and not otherwise on the ambient group.

\section{2-subsets attaining the universal bounds}

We now determine when the exact subindices of a $2$-subset attain the
universal lower and upper bounds. Let $G$ be a finite group of order $n$,
let $A=\{a,b\}$ with $a\ne b$, and put
\[
x=a^{-1}b,\quad
m=\operatorname{ord}(x),\quad
H=\langle x\rangle,\quad
q=|G:H|.
\]
Thus
\[
n=mq.
\]

\subsection{Two arithmetic lemmas}

We first record two elementary arithmetic facts that will be used repeatedly.

\begin{lemma}\label{lem:ceiling-equation}
Let $n=mq$, where $m,q$ are positive integers. Then
\[
q\left\lceil\frac{m}{3}\right\rceil
=
\left\lceil\frac{n}{3}\right\rceil
\]
if and only if one of the following conditions holds:
\begin{enumerate}
\item $m\equiv0\pmod3$;
\item $m\equiv1\pmod3$ and $q=1$;
\item $m\equiv2\pmod3$ and $q\in\{1,2\}$.
\end{enumerate}
\end{lemma}

\begin{proof}
Write $m=3s+r$, where $r\in\{0,1,2\}$.

If $r=0$, then
\[
q\left\lceil\frac{m}{3}\right\rceil=sq
=
\left\lceil\frac{3sq}{3}\right\rceil,
\]
so the equality holds for every $q$.

If $r=1$, then the equality becomes
\[
(s+1)q=sq+\left\lceil\frac{q}{3}\right\rceil,
\]
or equivalently
\[
q=\left\lceil\frac{q}{3}\right\rceil.
\]
This holds precisely when $q=1$.

If $r=2$, the equality becomes
\[
(s+1)q=sq+\left\lceil\frac{2q}{3}\right\rceil,
\]
or equivalently
\[
q=\left\lceil\frac{2q}{3}\right\rceil.
\]
This holds precisely when $q\in\{1,2\}$.
\end{proof}

\begin{lemma}\label{lem:floor-equation}
Let $n=mq$, where $m,q$ are positive integers. Then
\[
q\left\lfloor\frac{m}{2}\right\rfloor
=
\left\lfloor\frac{n}{2}\right\rfloor
\]
if and only if $m$ is even, or $m$ is odd and $q=1$.
\end{lemma}

\begin{proof}
If $m=2s$, then
\[
q\left\lfloor\frac{m}{2}\right\rfloor=sq
=
\left\lfloor\frac{2sq}{2}\right\rfloor,
\]
so the equality holds for every $q$.

If $m=2s+1$, then
\[
\left\lfloor\frac{n}{2}\right\rfloor
=
sq+\left\lfloor\frac{q}{2}\right\rfloor.
\]
Hence the required equality is equivalent to
\[
\left\lfloor\frac{q}{2}\right\rfloor=0,
\]
which holds precisely when $q=1$.
\end{proof}

\subsection{Attaining the lower bound}

Recall that
\[
|A^{-1}A|=|AA^{-1}|=
\begin{cases}
2, & m=2,\\
3, & m\ge3.
\end{cases}
\]
Thus the universal lower bounds for the right and left lower subindices are
the same.

\begin{theorem}\label{thm:lower-bound-2subset}
Let $G$ be a finite group and let $A=\{a,b\}$ with $a\ne b$. Put
$m=\operatorname{ord}(a^{-1}b)$ and
$q=|G:\langle a^{-1}b\rangle|$. Then $A$ attains the universal lower bound,
that is,
\[
|G:A|^-=\left\lceil\frac{|G|}{|A^{-1}A|}\right\rceil.
\]
if and only if either $m=2$, or $m\ge3$ and one of the following conditions
holds:
\begin{enumerate}
\item $m\equiv0\pmod3$;
\item $m\equiv1\pmod3$ and $q=1$;
\item $m\equiv2\pmod3$ and $q\le2$.
\end{enumerate}
\end{theorem}

\begin{proof}
If $m=2$, then $|A^{-1}A|=2$ and $n=2q$. Hence
\[
|G:A|^-=q\left\lceil\frac{2}{3}\right\rceil=q
=
\left\lceil\frac{n}{2}\right\rceil.
\]
Suppose now that $m\ge3$. Then $|A^{-1}A|=3$, and by
Theorem~\ref{thm:exact-subindices},
\[
|G:A|^-=q\left\lceil\frac{m}{3}\right\rceil.
\]
Therefore $A$ attains the lower bound if and only if
\[
q\left\lceil\frac{m}{3}\right\rceil
=
\left\lceil\frac{n}{3}\right\rceil.
\]
The result now follows from Lemma~\ref{lem:ceiling-equation}.
\end{proof}

\subsection{Attaining the upper bound}

Since $|A|=2$, the universal upper bound is
\[
\left\lfloor\frac{|G|}{2}\right\rfloor.
\]

\begin{theorem}\label{thm:upper-bound-2subset}
Let $G$ be a finite group and let $A=\{a,b\}$ with $a\ne b$. Put
$m=\operatorname{ord}(a^{-1}b)$. Then $A$ attains the universal upper bound,
that is,
\[
|G:A|^+=\left\lfloor\frac{|G|}{2}\right\rfloor,
\]
if and only if either
\begin{enumerate}
\item $m$ is even; or
\item $m$ is odd and $\langle a^{-1}b\rangle=G$.
\end{enumerate}
\end{theorem}

\begin{proof}
Put
\[
q=|G:\langle a^{-1}b\rangle|.
\]
By Theorem~\ref{thm:exact-subindices},
\[
|G:A|^+=q\left\lfloor\frac{m}{2}\right\rfloor.
\]
Since $|G|=mq$, the required equality is
\[
q\left\lfloor\frac{m}{2}\right\rfloor
=
\left\lfloor\frac{mq}{2}\right\rfloor.
\]
By Lemma~\ref{lem:floor-equation}, this holds precisely when $m$ is even,
or when $m$ is odd and $q=1$.
\end{proof}

\subsection{Index-stable 2-subsets attaining both bounds}

We can now characterize the index-stable $2$-subsets whose common subindex
simultaneously attains the universal lower and upper bounds.

\begin{theorem}\label{thm:stable-both-bounds}
Let $G$ be a finite group and let $A=\{a,b\}$ with $a\ne b$. Then $A$ is
index-stable and attains both universal bounds if and only if one of the
following conditions holds:
\begin{enumerate}
\item $\operatorname{ord}(a^{-1}b)=2$;
\item $\langle a^{-1}b\rangle=G$ and
$\operatorname{ord}(a^{-1}b)\in\{3,4,5,7\}$.
\end{enumerate}

In case \textup{(i)}, the common subindex is $|G|/2$. In case \textup{(ii)}, the
common subindex is respectively
\[
1,\ 2,\ 2,\ 3
\]
for $|G|=3,4,5,7$.
\end{theorem}

\begin{proof}
By Theorem~\ref{thm:2subset-stability}, index stability is equivalent to
\[
m=\operatorname{ord}(a^{-1}b)\in\{2,3,4,5,7\}.
\]

If $m=2$, Theorems~\ref{thm:lower-bound-2subset}
and~\ref{thm:upper-bound-2subset} show that both bounds are attained for every
$q$, and the common value is $|G|/2$.

Suppose now that $m\in\{3,4,5,7\}$. If $m=3$, attainment of the upper bound
requires $q=1$. If $m=4$, attainment of the lower bound requires $q=1$.
If $m=5$, attainment of the upper bound requires $q=1$, and if $m=7$,
attainment of either bound requires $q=1$.
Thus in every case $q=1$,
or equivalently $\langle a^{-1}b\rangle=G$.
The values of the common subindex follow from
Theorem~\ref{thm:exact-subindices}.
\end{proof}

\begin{corollary}\label{cor:cyclic-both-bounds}
Let $A=\{a,b\}\subseteq\mathbb Z_n$, where $a\ne b$, and put
\[
d=\gcd(n,a-b).
\]
Then $A$ is index-stable and attains both universal bounds if and only if
either
\begin{enumerate}
\item $n$ is even and $d=n/2$, in which case the common subindex is $n/2$;
or
\item $n\in\{3,4,5,7\}$ and $d=1$, in which case the common subindex is
respectively $1,2,2,3$.
\end{enumerate}
\end{corollary}

\subsection{Strongly index-unstable 2-subsets}

We now consider the opposite extremal situation, where a $2$-subset attains
both universal bounds but the two attained values are distinct.

Let $G$ be a finite group and let $A=\{a,b\}$ with $a\ne b$. We call $A$
\emph{strongly index-unstable} if
\[
|G:A|^-=\left\lceil\frac{|G|}{|A^{-1}A|}\right\rceil
<
|G:A|^+=\left\lfloor\frac{|G|}{2}\right\rfloor.
\]
The following theorem and its corollary give the complete characterization;
the cyclic case is stated in terms of $d=\gcd(n,a-b)$.

\begin{theorem}\label{thm:strongly-unstable-2subset}
Let $G$ be a finite group such that
\[
|G|\notin\{2,3,4,5,7\},
\]
and let $A=\{a,b\}$ with $a\ne b$. Put
\[
m=\operatorname{ord}(a^{-1}b),
\quad
q=|G:\langle a^{-1}b\rangle|.
\]
Then $A$ is strongly index-unstable if and only if $m\ge3$ and one of the
following conditions holds:
\begin{enumerate}
\item $q=1$;
\item $q=2$ and $m\equiv0,2\pmod6$;
\item $q\geq3$ and $m\equiv0\pmod6$.
\end{enumerate}
\end{theorem}

\begin{proof}
If $m=2$, then both subindices equal $|G|/2$, so $A$ is not
strongly index-unstable. Suppose that $m\ge3$.

By Theorems~\ref{thm:lower-bound-2subset} and
\ref{thm:upper-bound-2subset}, if $q=1$, both bounds are attained,
giving condition (i).

Suppose now that $q>1$. The upper bound is attained if and only if
$m$ is even. Under this assumption, the lower bound is attained if
and only if either $3\mid m$, or $m\equiv2\pmod3$ and $q\le2$.
These possibilities yield conditions (ii) and (iii).

Finally, since $m\ge3$, the two bounds are
$\lceil |G|/3\rceil$ and $\lfloor |G|/2\rfloor$.
They are distinct precisely when $|G|\notin\{2,3,4,5,7\}$,
as required.
\end{proof}

The following corollary gives a partial answer, for $2$-subsets, to Problem~1 in
\cite{Hoo-Zn} on subindices and subfactors of $\mathbb Z_n$.

\begin{corollary}\label{cor:strongly-unstable-cyclic}
Let $A=\{a,b\}\subseteq\mathbb Z_n$, where $a\ne b$, and put
\[
d=\gcd(n,a-b).
\]
Then $A$ is strongly index-unstable if and only if
\[
n\notin\{2,3,4,5,7\},
\quad
d\le\frac{n}{3},
\]
and at least one of the following conditions holds:

\begin{enumerate}
\item $d=1$;
\item $d=2$ and $n/d\equiv0,2\pmod6$;
\item $d\ge3$ and $n/d\equiv0\pmod6$.
\end{enumerate}
\end{corollary}

\begin{proof}
Since $\operatorname{ord}(a-b)=n/d$ and
$|\mathbb Z_n:\langle a-b\rangle|=d$, the result follows
from Theorem~\ref{thm:strongly-unstable-2subset}.
\end{proof}

\section{Classification of finite 2-index-stable groups}

For a finite group $G$, put
\[
\omega(G)=\{\operatorname{ord}(g):g\in G\},
\quad
\pi(G)=\{p\text{ prime}:p\mid |G|\},
\]
and
\[
\mathcal S=\{G\text{ finite}:\omega(G)\subseteq\{1,2,3,4,5,7\}\}.
\]
By Theorem~\ref{thm:2subset-stability}, a finite group $G$ is
$2$-index-stable if and only if $G\in\mathcal S$.

Moreover, the class $\mathcal S$ is closed under taking subgroups
and quotients. Indeed, if $H\le G$, then
$\omega(H)\subseteq\omega(G)$; and if $N\trianglelefteq G$,
the order of $gN$ divides the order of $g$.

For the solvable case, we use Higman's theorem on finite solvable
groups whose elements have prime-power orders
\cite[Theorem~1]{Higman}.

For the nonsolvable case, we use
\cite[Theorem~1.7]{Cameron-Maslova}.
Recall that the Gruenberg--Kegel graph of $G$ has vertex set
$\pi(G)$, with distinct primes $p,q$ adjacent if and only if
$G$ has an element of order $pq$.
Here $O_2(G)$ denotes the largest normal $2$-subgroup of $G$.

\begin{theorem}\label{thm:structure-2-index-stable}
A nontrivial finite group $G$ is $2$-index-stable if and only if one of the
following holds:
\begin{enumerate}
\item $G$ is a $p$-group satisfying
\[
\exp(G)\le4
\]
for $p=2$, or
\[
\exp(G)=p
\]
for $p\in\{3,5,7\}$;

\item $G$ is a Frobenius group
\[
G=P\rtimes H,
\]
where $P$ is a $p$-group satisfying the exponent condition in \textup{(i)},
\[
H\in\{C_2,C_3,C_4,C_5,C_7,Q_8\},
\]
$(p,|H|)=1$, and the action of $H$ on $P$ is fixed-point-free;

\item $G$ is a $2$-Frobenius group with a normal series
\[
1\mathrel{\trianglelefteq}P
\mathrel{\trianglelefteq}M
\mathrel{\trianglelefteq}G
\]
such that $P$ is a $p$-group, $M$ is a Frobenius group with kernel $P$,
and $G/P$ is a Frobenius group with kernel $M/P$, with
\[
M/P\cong C_q,
\quad
G/M\cong C_{p^a},
\]
where
\[
(p^a,q)\in\{(2,3),(2,5),(4,5),(2,7),(3,7)\},
\]
and a Sylow $p$-subgroup $S$ of $G$ satisfies
\[
\exp(S)\le4
\]
for $p=2$, or
\[
\exp(S)=3
\]
for $p=3$;

\item
$G\cong A_6$ or $G\cong\operatorname{PSL}(2,7)$;

\item
$G/O_2(G)\cong A_5\cong\operatorname{PSL}(2,4)$,
and, if $O_2(G)\ne1$, then $O_2(G)$ is a direct product of minimal normal
subgroups of $G$, each of order $16$ and, as an $A_5$-module, isomorphic
to the natural $GF(4)$-module for $\operatorname{PSL}(2,4)$;

\item
$G\cong\operatorname{PSL}(3,4)$.
\end{enumerate}
\end{theorem}

\begin{proof}
By the element-order criterion above, it suffices to classify
the groups in $\mathcal S$. We treat solvable groups in
Steps~1--3 and nonsolvable groups in Step~4, proving both
implications in each case.

\smallskip
\noindent\textbf{Step 1: the solvable case and Higman's theorem.}
Let $G$ be solvable.

If $G$ is a $p$-group, then all its element orders are powers of $p$.
Thus $G\in\mathcal S$ if and only if either $p=2$ and
$\exp(G)\le4$, or $p\in\{3,5,7\}$ and $\exp(G)=p$.
This is exactly condition \textup{(i)}.

Suppose now that $G\in\mathcal S$ is solvable and $|G|$ is
divisible by at least two distinct primes.
By \cite[Theorem~1]{Higman}, there is a prime $p$
such that $P=O_p(G)$, the largest normal $p$-subgroup
of $G$, is nontrivial and proper, and $G/P$ is either
($\alpha$) cyclic of prime-power order coprime to $p$;
($\beta$) a generalized quaternion group, with $p$ odd; or
($\gamma$) a group of order $p^a q^b$ with cyclic Sylow subgroups,
where $q=kp^a+1$ is prime and $a,b,k\ge1$.
Since $P,G/P\in\mathcal S$, we have $p\in\{2,3,5,7\}$,
and $P$ satisfies the exponent condition in \textup{(i)}.
We consider the first two cases in Step~2 and the third in Step~3.

\smallskip
\noindent\textbf{Step 2: the Frobenius case.}
If $G/P$ is cyclic of prime-power order, then $G/P\in\mathcal S$ forces
\[
G/P\cong C_2,\ C_3,\ C_4,\ C_5,\ \text{or }C_7.
\]
If $G/P$ is generalized quaternion, then it is a $2$-group in $\mathcal S$;
a generalized quaternion group of order greater than $8$ contains an element
of order $8$, hence $G/P\cong Q_8$ (and $p$ is odd). In both cases
$G/P\cong H$ with $H\in\{C_2,C_3,C_4,C_5,C_7,Q_8\}$ and $(p,|H|)=1$. Thus $P$
is a normal Hall subgroup, and by the Schur--Zassenhaus theorem
\[
G=P\rtimes H.
\]
Let $1\ne h\in H$. If $h$ centralized some $1\ne y\in P$, then $yh$ would be
an element of order divisible by two distinct primes, contradicting
$G\in\mathcal S$. Hence $C_P(h)=1$ for every $1\ne h\in H$, i.e.\ the action is
fixed-point-free and $G$ is a Frobenius group with kernel $P$ and complement
$H$. This gives \textup{(ii)}.

Conversely, let $G=P\rtimes H$ be as in \textup{(ii)}. Every element of $G$
lies in $P$ or in a conjugate of the Frobenius complement $H$, so
$\omega(G)=\omega(P)\cup\omega(H)\subseteq\{1,2,3,4,5,7\}$ by the exponent
condition on $P$ and the list of $H$.

\smallskip
\noindent\textbf{Step 3: the $2$-Frobenius case.}
Here $|G/P|=p^aq^b$, the Sylow subgroups of $G/P$ are cyclic, and
$q=kp^a+1$; in particular $q$ is odd. Since $G/P\in\mathcal S$, its cyclic
Sylow $q$-subgroup has order $q^b\in\{3,5,7\}$, so $b=1$, and its cyclic
Sylow $p$-subgroup has order $p^a\in\{2,3,4,5,7\}$. As $p\ne q$,
$p,q\in\{2,3,5,7\}$ and $p^a\mid q-1$, the only possibilities are
\[
(p^a,q)\in\{(2,3),(2,5),(4,5),(2,7),(3,7)\}.
\]
Since $q>p^a$, the number of Sylow $q$-subgroups of $G/P$ divides $p^a$ and is
congruent to $1$ modulo $q$, hence equals $1$. Let $M/P$ be the normal Sylow
$q$-subgroup of $G/P$; then $M/P\cong C_q$ and $G/M\cong C_{p^a}$. By the
Schur--Zassenhaus theorem $M=P\rtimes C_q$, and, as in Step~2, the action of
$C_q$ on $P$ is fixed-point-free, so $M$ is a Frobenius group with kernel
$P$.
Again by the Schur--Zassenhaus theorem, $G/P$ has a complement
to $M/P$, isomorphic to $C_{p^a}$.
No nonidentity element of this complement can centralize a
nonidentity element of $M/P$ in $G/P$
(it would produce an element of order divisible by $pq$),
so $G/P$ is a Frobenius group with kernel $M/P$. Thus $G$
is a $2$-Frobenius group. If $S$ is a Sylow $p$-subgroup of $G$, then
$S\cap M=P$, so $S/P\cong G/M\cong C_{p^a}$, and the exponent condition on
$S$ follows since $S\le G$ and $G\in\mathcal S$. This gives \textup{(iii)}.

Conversely, let $G$ be as in \textup{(iii)}.
Elements of $P$ and elements of $G\setminus M$ are $p$-elements:
if $x\in G\setminus M$, then $xP$ lies outside the Frobenius kernel $M/P$ of $G/P$
and is therefore conjugate to an element of a complement of order $p^a$. Hence
$\langle x\rangle P$ is a $p$-group.
By Sylow's theorem, their orders divide $\exp(S)$, so they lie
in $\{1,2,4\}$ ($p=2$) or $\{1,3\}$ ($p=3$).
Elements of $M\setminus P$ lie in
conjugates of a Frobenius complement of $M$, of order $q\in\{3,5,7\}$. Hence
$\omega(G)\subseteq\{1,2,3,4,5,7\}$.

Steps~1--3 show that a solvable group is $2$-index-stable if and only if it is
one of the groups \textup{(i)}--\textup{(iii)}.

\smallskip
\noindent\textbf{Step 4: the nonsolvable case.}
Let $G\in\mathcal S$ be nonsolvable. Its Gruenberg--Kegel graph
has no edges, since $G$ has no element of order $pq$ for
distinct primes $p,q$. Moreover, $\pi(G)\subseteq\{2,3,5,7\}$.
By \cite[Theorem~1.7]{Cameron-Maslova}, $G$ satisfies one
of cases \textup{(3)}--\textup{(6)} of that theorem.

In case \textup{(3)} of that theorem, $\operatorname{PSL}(2,17)$ is excluded
because $17\in\pi(G)$ (its order is $2^4\cdot3^2\cdot17$), and $M_{10}$ is
excluded because $\omega(M_{10})=\{1,2,3,4,5,8\}$; this leaves $A_6$ and
$\operatorname{PSL}(2,7)$.

In case \textup{(4)}, the possibility $G/O_2(G)\cong\operatorname{PSL}(2,8)$ is excluded,
since this quotient would belong to $\mathcal S$, whereas
$\omega(\operatorname{PSL}(2,8))=\{1,2,3,7,9\}$.
Hence $G/O_2(G)\cong\operatorname{PSL}(2,4)\cong A_5$,
with the module structure of $O_2(G)$ described in \textup{(v)}.

Case \textup{(5)} gives $\operatorname{PSL}(3,4)$. Case \textup{(6)} is
impossible, since $13\mid|Sz(8)|=2^6\cdot5\cdot7\cdot13$ and
$31\mid|Sz(32)|=2^{10}\cdot5^2\cdot31\cdot41$, so these quotients (and hence $G$)
would have prime divisors outside $\{2,3,5,7\}$.

This yields \textup{(iv)}, \textup{(v)}, \textup{(vi)} of the theorem.

Conversely,
\[
\omega(A_6)=\{1,2,3,4,5\},\quad
\omega(\operatorname{PSL}(2,7))=\{1,2,3,4,7\},\quad\text{and}\quad
\omega(\operatorname{PSL}(3,4))=\{1,2,3,4,5,7\},
\]
so these groups lie in $\mathcal S$.
Let $G$ be as in \textup{(v)}.
By \cite[Theorem~1.7]{Cameron-Maslova}, the Gruenberg--Kegel graph of $G$ has no edges,
and $O_2(G)$ is elementary abelian.
If $x\in G$ has $2$-power order, then its image in
$G/O_2(G)\cong A_5$ has order at most $2$, so $x^2\in O_2(G)$ and $x^4=1$.
If $x$ has odd order, then $\langle x\rangle\cap O_2(G)=1$, so the order of
$x$ equals that of its image in $A_5$, namely $1$, $3$, or $5$.
Since no element order is divisible by two primes, $\omega(G)\subseteq\{1,2,3,4,5\}$, and
$G\in\mathcal S$.

Combining Steps~1--4: a solvable $G$ is $2$-index-stable exactly when it is
in \textup{(i)}--\textup{(iii)}, and a nonsolvable $G$ exactly when it is in
\textup{(iv)}--\textup{(vi)}.
This proves the theorem.
\end{proof}

For example, $S_3$, the dihedral group of order $10$, $A_4$ and $C_3^2\rtimes Q_8$,
with the fixed-point-free action of $Q_8$ on $C_3^2$, are Frobenius
groups belonging to \textup{(ii)}, the group $S_4$ is a $2$-Frobenius group
belonging to \textup{(iii)}, and $A_5$ belongs to \textup{(v)} with
$O_2(G)=1$.
The dihedral group of order $8$ belongs to \textup{(i)}.
Another example of \textup{(v)} is
$GF(4)^2\rtimes\operatorname{SL}(2,4)$,
with the natural action, a group of order $960$.

\end{document}